\documentclass[11pt]{article}
\usepackage{epic,eepic}
\usepackage{amsmath,amssymb,amsfonts}
\begin{document}
\newtheorem{theorem}{Theorem}[section]
\newtheorem{corollary}[theorem]{Corollary}
\newtheorem{lemma}[theorem]{Lemma}
\newtheorem{remark}[theorem]{Remark}
\newtheorem{proposition}[theorem]{Proposition}
\newtheorem{definition}[theorem]{Definition}
\def\emptyset{\varnothing}
\def\setminus{\smallsetminus}
\def\Irr{{\mathrm{Irr}}}
\def\Rep{{\mathrm{Rep}}}
\def\End{{\mathrm{End}}}
\def\Tube{{\mathrm{Tube}}}
\def\Vec{{\mathrm{Vec}}}
\def\loc{{\mathrm{loc}}}
\def\opp{{\mathrm{opp}}}
\def\id{{\mathrm{id}}}
\def\A{{\mathcal{A}}}
\def\B{{\mathcal{B}}}
\def\C{{\mathcal{C}}}
\def\D{{\mathcal{D}}}
\def\E{{\mathcal{E}}}
\def\X{{\mathcal{X}}}
\def\CC{{\mathbb{C}}}
\def\N{{\mathbb{N}}}
\def\Q{{\mathbb{Q}}}
\def\R{{\mathbb{R}}}
\def\Z{{\mathbb{Z}}}
\def\a{{\alpha}}
\def\be{{\beta}}
\def\de{{\delta}}
\def\e{{\varepsilon}}
\def\si{{\sigma}}
\def\la{{\lambda}}
\def\th{{\theta}}
\def\lan{{\langle}}
\def\ran{{\rangle}}
\def\isom{{\cong}}
\newcommand{\Hom}{\mathop{\mathrm{Hom}}\nolimits}
\def\qed{{\unskip\nobreak\hfil\penalty50
\hskip2em\hbox{}\nobreak\hfil$\square$
\parfillskip=0pt \finalhyphendemerits=0\par}\medskip}
\def\proof{\trivlist \item[\hskip \labelsep{\bf Proof.\ }]}
\def\endproof{\null\hfill\qed\endtrivlist\noindent}

\title{The relative Drinfeld commutant \\
of a fusion category and $\alpha$-induction}
\author{
{\sc Yasuyuki Kawahigashi}\\
{\small Graduate School of Mathematical Sciences}\\
{\small The University of Tokyo, Komaba, Tokyo, 153-8914, Japan}
\\[0,05cm]
{\small and}
\\[0,05cm]
{\small Kavli IPMU (WPI), the University of Tokyo}\\
{\small 5-1-5 Kashiwanoha, Kashiwa, 277-8583, Japan}\\
{\small e-mail: {\tt yasuyuki@ms.u-tokyo.ac.jp}}}
\maketitle{}
\begin{abstract}
We establish a correspondence among simple
objects of the relative commutant of a full fusion subcategory
in a larger fusion category in the sense of Drinfeld,
irreducible half-braidings of objects in the larger fusion category 
with respect to the fusion subcategory, and
minimal central projections in the relative tube algebra.
Based on this, we explicitly compute certain relative 
Drinfeld commutants of
fusion categories arising from $\alpha$-induction for
braided subfactors.  We present examples arising from
chiral conformal field theory.
\end{abstract}

\section{Introduction}

The notion of a Drinfeld center has been studied well within the
Jones theory of subfactors \cite{J}.  Around 1990, Ocneanu 
realized that his construction of the asymptotic inclusion
$M\vee (M'\cap M_\infty)\subset M_\infty$
from a hyperfinite type II$_1$
subfactor $N\subset M$ with finite index and finite depth gives
an operator algebraic counterpart of the Drinfeld center
construction, also called the Drinfeld double or
the ``quantum double''.  We refer the reader to \cite{EK1} for
Ocneanu's theory and related results, and to \cite{I} for
an approach based on Longo's sector theory \cite{L1}, \cite{L2}.

The notion of a Drinfeld center is similar to that of a usual
center of an algebra, as the name shows. Henriques \cite{H}
recently studies the Drinfeld version of a (relative/double) commutant
of a fusion category.  In this paper, we study the notion
of the relative commutant of a full fusion subcategory in another
fusion category and clarify its relations to (the relative
version of) Ocneanu's tube algebra and half-braidings along
the line of \cite{I}.  We have made several computations of
the Drinfeld centers for certain fusion categories arising from
$\a$-induction in \cite{BEK3}.  (Here $\a$-induction is a
certain induction machinery originally introduced for an
extension of a chiral conformal field theory in \cite{LR}.)
In this paper, we make analogous computations for
the relative commutants of these fusion categories arising
from $\a$-induction.  Our methods are similar to those in
\cite{BEK3} and rely on half-braidings arising from 
relative braidings studied in \cite{BE3}.

We refer the reader to \cite{EK2} for a general theory
of subfactors and to \cite{K} for a review on subfactor theory,
category theory, and conformal field theory.

The author would like to thank the Isaac Newton Institute for
Mathematical Sciences, Cambridge, for support and hospitality during
the program ``Operator Algebras: Subfactors and their Applications''
where work on this paper was undertaken.
This work was partially supported by EPSRC Grant Number EP/K032208/1.
This work was also partially supported by a grant from the
Simons Foundation and Grants-in-Aid for Scientific Research
15H02056.  The author thanks M. Izumi and R. Longo for comments
on this work.  The author also thanks the referee for helpful
report.

\section{The relative tube algebra and the relative Drinfeld
commutant}
\label{tube}

Let $\D$ be a unitary fusion category and $\C$ its full
subcategory.  (We consider only {\sl unitary} fusion categories
in this paper.)  We may and do assume that $\D$ is realized as a
category of endomorphisms of a type III factor $M$ with 
finite index.  We fix a representative in each equivalence
class of simple objects of $\D$ and let $\Irr(\D)$ be the
set of such representatives.  The set $\Irr(\C)$ is a
subset of $\Irr(\D)$ consisting of objects in $\C$.
We assume that the identity morphism is in $\Irr(\C)$ and
denote it by $\id$.
For an object $\la$ in $\D$, we write $d(\la)$ for its
dimension, the square root of the minimal index of the
subfactor $\la(M)\subset M$.
We set $\dim \C$ to be the square sum of the dimensions of
the equivalence classes of the simple objects of $\C$.
(That is, we have $\dim \C=\sum_{\la\in\Irr(\C)} d(\la)^2$.)
We similarly define $\dim \D$.
For an object $\la$ in $\D$, we write $\phi_\la$ for
its standard left inverse.  (See \cite[page 238]{L1}
for a left inverse.)

We first have a notion of a half-braiding of an object in
$\D$ with respect to $\C$ as in \cite[Definition 2.2]{BEK3}
which is a slight generalization of \cite[Definition 4.2]{I}.

\begin{definition}{\rm
Let $\si$ be an object of $\D$.  We call a family of
unitary intertwiners $\E_\si=\{\E_\si(\be)\}_{\be\in\Irr(\C)}$
a half-braiding of $\si$ with respect to $\C$ if it
satisfies the following two conditions.

(1) We have $\E_\si(\be)\in\Hom(\si\be,\be\si)$ for all
$\be\in\Irr(\C)$.

(2) For $\be_1,\be_2,\be_3\in\Irr(\C)$, we have
\[
X\E_\si(\be_3)=\be_1(\E_\si(\be_2))\E_\si(\be_1)\si(X)
\]
for any $X\in\Hom(\be_3,\be_1\be_2)$.

Two pairs $(\si,\E_\si)$, $(\si',\E'_{\si'})$ of
objects $\si,\si'$ in $\D$ with respective half-braidings
$\E_\si, \E'_{\si'}$ are said to be equivalent of there
is a unitary intertwiner $u\in\Hom(\si',\si)$ with
\[
\E_\si(\be)=\be(u)\E'_{\si'}(\be)u^*
\]
for all $\be\in\Irr(\C)$.
}\end{definition}

The equality in (2) is called the braiding-fusion equation.

In order to distinguishing different half-braidings for the
same $\si$, we use the notation $\E^\a_\si$ as in \cite{I},
where $\a$ denotes an index.

The objects in $\D$ with half-braidings
with respect to $\C$ make a fusion category as in
\cite[Definition 2.1]{H}.  We call it the
relative Drinfeld commutant of $\C$ in $\D$ and
write $\C'\cap \D$ for it.
(It is called simply the commutant of $\C$ in $\D$
in \cite[Section 2.1]{H}, but we add the word
``Drinfeld'' in order to emphasize that this is
different from a usual relative commutant of a
subalgebra.)  
Note that the conjugate half-braiding 
$\E^{\bar\a}_{\bar\si}$ of a half-braiding 
$\{\E^\a_\si(\be)\}_{\be\in\Irr(\C)}$ is
defined by $\E^{\bar\a}_{\bar\si}(\be)=
d(\si)R^*_\si(\E^\a_\si(\be)^*\be(\bar R_\si))$
as in \cite[Theorem 4.6 (iv)]{I}. 
For half-braidings 
$\{\E_\si(\be)\}_{\be\in\Irr(\C)}$ and
$\{\E_{\si'}(\be)\}_{\be\in\Irr(\C)}$, the fusion
product is given by
$\{\E_\si(\be)\si(\E_{\si'}(\be))\}_{\be\in\Irr(\C)}$.
For half-braidings 
$\{\E_\si(\be)\}_{\be\in\Irr(\C)}$ and
$\{\E_{\si'}(\be)\}_{\be\in\Irr(\C)}$, an intertwiner
from the former to the latter is given by
$X\in\Hom(\si,\si')$ satisfying
$\E_{\si'}(\be)X=\be(X)\E_\si(\be)$ for all $\be\in\Irr(\C)$.

Obviously, the fusion category $\C'\cap \C$, the Drinfeld
center of $\C$, is a full subcategory
of $\C'\cap \D$, but note that $\D'\cap\D$ is not a
full subcategory of $\C'\cap\D$.
If $\C$ is a trivial category $\Vec$
of finite dimensional Hilbert spaces, then $\C'\cap \D$
is simply $\D$.

We choose a representative $\{\E^\a_\si\}$
from each equivalence class of simple objects in 
$\C'\cap\D$ and write $\Irr(\C'\cap\D)$ for the
set consisting of them.

We next introduce the relative tube algebra which generalizes
a notion of Ocneanu's tube algebra studied in
\cite[Section 3]{EK1}, \cite[Section 3]{I}.

\begin{definition}{\rm
We set the relative tube algebra $\Tube(\C,\D)$ to be
\[
\bigoplus_{\la,\nu\in\Irr(\D),\mu\in\Irr(\C)}
\Hom(\la\mu,\mu\nu)
\]
as a linear space.  We define its algebra structure and
$*$-structure by the same formulas as in \cite[page 134]{I}.
}\end{definition}

As in \cite[Section 3]{I}, we write $(\la\mu|X|\mu\nu)$
for $X\in\Hom(\la\mu,\mu\nu)$ for indicating which
space $X$ belongs to.

For $(\la\mu|X|\mu\nu)\in\Tube(\C,\D)$, we set
\[
\varphi_{\C,\D}((\la\mu|X|\mu\nu))=
d(\la)^2\de_{\la,\nu}\de_{\mu,0}X.
\]

Note that $X$ on the right hand side is a scalar
in $\Hom(\la,\la)$ if the right hand side does not vanish.
We remark that
$\Tube(\C,\D)$ is a finite dimensional $C^*$-algebra
as exactly in \cite[Proposition 3.2]{I}.

Now we follow the arguments in \cite[page 146]{I}.
Let $\{\E^\a_\si(\be)\}_{\be\in\Irr(\C)}$ be a half-braiding
of an object $\si$ in $\D$ with respect to $\C$ where we have
$[\si]=\bigoplus_{\la\in\Irr(\D)} n_\la [\la]$.  We fix an
orthonormal basis $\{W_\si(\la)_i\}_{i=1}^{n_\la}\subset
\Hom(\la,\si)$ and put
\[
\E^\a_\si(\be)_{(\la,i),(\mu,j)}=
\be(W_\si(\mu)^*_j)\E^\a_\si(\be)W_\si(\la)_i\in
\Hom(\la\be,\be\mu)\subset\Tube(\C,\D),
\]
where $\be$ is in $\Irr(\C)$.  We then have
\[
\E^\a_\si(\be)=\sum_{\la,\mu\in\Irr(\D)}
\sum_{i=1}^{n_\la}\sum_{j=1}^{n_\mu}
\be(W_\si(\mu)_j)\E^\a_\si(\be)_{(\la,i),(\mu,j)}W_\si(\la)^*_i.
\]
We next put
\[
e(\E_\si^\a)_{(\la,i),(\mu,j)}=\frac{d(\si)}
{(\dim \C) d(\la)^{1/2}d(\mu)^{1/2}}
\sum_{\be\in\Irr(\C)} d(\be)
(\la\be|\E^\a_\si(\be)_{(\la,i),(\mu,j)}|\be\mu)\in\Tube(\C,\D).
\]

The following is a slight generalization of \cite[Lemma 4.7]{I}
with essentially the same proof.

\begin{lemma}
For $e(\E_\si^\a)_{(\la,i),(\mu,j)}$ as above and
$X\in\Hom(\nu\si,\si\tau)\subset\Tube(\C,\D)$ where
$\si\in\Irr(\C)$ and $\nu,\tau\in\Irr(\D)$, we have the following.

(1) We have 
$e(\E_\si^\a)^*_{(\la,i),(\mu,j)}=e(\E_\si^\a)_{(\mu,j),(\la,i)}$.

(2) We have
\begin{align*}
&e(\E_\si^\a)_{(\la,i),(\mu,j)}(\nu\si|X|\si\tau)\\
=&\de_{\mu,\nu}
\frac{d(\si)d(\tau)^{1/2}}{d(\mu)^{1/2}}
\sum_k\phi_\si(X\E_\si^\a(\si)^*_{(\mu,j),(\tau,k)})
e(\E_\si^\a)_{(\la,i),(\tau,k)}.
\end{align*}

(3) We have 
\begin{align*}
&(\nu\si|X|\si\tau)e(\E_\si^\a)_{(\la,i),(\mu,j)}\\
=&\de_{\tau,\la}
\frac{d(\si)d(\nu)^{1/2}}{d(\la)^{1/2}}
\sum_k\phi_\nu(\E_\si^\a(\si)^*_{(\nu,k),(\la,i)}X)
e(\E_\si^\a)_{(\nu,k),(\mu,j)}.
\end{align*}
\end{lemma}

The following is also a slight generalization of \cite[Corollary 4.8]{I}
with essentially the same proof.

\begin{corollary}
Let $\E^\a_\si\in\Irr(\C'\cap\D)$.  Then
$z(\E^\a_\si)=\sum_{\la,i} e(\E^\a_\si)_{(\la,i),(\la,i)}$
is in the center of $\Tube(\C,\D)$.
\end{corollary}

The following is also a slight generalization of \cite[Lemma 4.9]{I}
with essentially the same proof.

\begin{lemma}
In the above setting, we have the following.

(1) We have $\sum_{\E^\a_\si\in\Irr(\C,\D)} d(\E^\a_\si)^2=
\dim \C\; \dim \D$.

(2) We have $\varphi_{\C,\D}(z(\E^\a_\si))=d(\si)^2/\dim \C$.

(3) We have $\varphi_{\C,\D}(1)=\dim \D$.
\end{lemma}

The following is again a slight generalization of
\cite[Theorem 4.10]{I} with essentially the same proof.

\begin{theorem}
Let $e(\E_\si^\a)_{(\la,i),(\mu,j)}$ be as above. 
Then we have the following.

(1) The system
$\{e(\E_\si^\a)_{(\la,i),(\mu,j)}\}_{(\la,i),(\mu,j)}$
is a system of matrix units of a simple component of
$\Tube(\C,\D)$.

(2) The operators $\{z(\E_\si^\a)\}_{\E^\a_\si\in\Irr(\C'\cap\D)}$
are mutually orthogonal minimal central projections of
$\Tube(\C,\D)$ with
$\sum_{\E^\a_\si\in\Irr(\C'\cap\D)}z(\E^\a_\si)=1$.
\end{theorem}

\section{A half-braiding and $\eta$-extension}

We keep the notation of Section \ref{tube}.
Let $M\otimes M^{\opp}\subset R$ be the Longo-Rehren subfactor
\cite{LR}
corresponding to $\Irr(\C)$ with the dual canonical endomorphism 
$\Theta=\bigoplus_{\la\in\Irr(\C)}\la\otimes\la^\opp$
and the inclusion map $\iota_{LR}: M\otimes M^\opp\hookrightarrow R$.
Here we use the anti-isomorphism $j:M\to M^\opp$ and
$\si^\opp=j\cdot\si\cdot j^{-1}$ for an endomorphism $\si$ of $M$
which is an endomorphism of $M^\opp$.  We have an isometry $V\in R$ with
$(M\otimes M^\opp)V=R$ and
$Vx=\iota_{LR}\cdot\bar\iota_{LR}(x)V$ for $x\in R$
since the Longo-Rehren subfactor has a finite index.

The following is a direct analogue of \cite[Theorem 4.1]{I}
and can be proved in the same way.
(The Longo-Rehren subfactor studied in \cite{I} is dual to the
one studied in \cite{BEK3} and here, but this difference is only
superficial.)

\begin{proposition}\label{pro}
(1) The set $\{\la\otimes\mu^\opp\}_{\la\in\Irr(\D),\mu\in\Irr(\C)}$ gives
mutually inequivalent $M\otimes M^\opp$-$M\otimes M^\opp$ sectors
and the sectors associated with the Longo-Rehren subfactor
$M\otimes M^{\opp}\subset R$ give its subset.

(2) The set $\{\iota_{LR}\cdot(\la\otimes\id^\opp)\}_{\la\in\Irr(\D)}$
gives mutually inequivalent $R$-$M\otimes M^\opp$ sectors.
We have
$[\iota_{LR}\cdot(\la\otimes\id^\opp)]=
[\iota_{LR}\cdot(\id\otimes\bar\la^\opp)]$ for $\la\in\Irr(\C)$,
and
\[
[\iota_{LR}\cdot(\la\otimes\mu^\opp)]=\bigoplus_{\nu\in\Irr(\D)}
N_{\bar\mu,\la}^\nu [\iota_{LR}\cdot(\nu\otimes\id^\opp)].
\]
for $\la\in\Irr(\D)$ and $\mu\in\Irr(\C)$.
\end{proposition}

Statement (1) above holds also true for $\mu\in\Irr(\D)$, but it
is important to  consider only $\mu\in\Irr(\C)$ for (2).

For a half-braiding $\E_\si$ of an object $\si$ in $\D$,
we set as follows as in \cite[page 139]{I}.

\begin{align*}
U(\si,\E_\si)&=\sum_{\la\in\Irr(\C)} W_\la(\E_\si(\la)\otimes1)
(\si\otimes\id^\opp)(W^*_\la),\\
U(\si^\opp,\E_\si)&=\sum_{\la\in\Irr(\C)} W_\la(1\otimes j((\E_\si(\la)))
(\id\otimes\si^\opp)(W^*_\la),
\end{align*}
where $\{W_\la\}_{\la\in\Irr(\C)}$ is a set of isometries in
$M\otimes M^\opp$ with $\sum_{\la\in\Irr(\C)}W_\la W_\la^*=1$
and $W_\la^*V\in(\iota_{LR},\iota_{LR}\cdot(\la\otimes\la^\opp))$.

We then define an $\eta$-extension of $\si\otimes\id$, an
endomorphism of $M\otimes M^\opp$, to $R$ as follows as in
\cite[Definition 2.3]{BEK3}.
\begin{align*}
\eta(\si,\E_\si)(a)&=(\si\otimes\id)(a),\quad a\in M\otimes M^\opp,\\
\eta(\si,\E_\si)(V)&=U(\si,\E_\si)^*V.
\end{align*}
This is indeed an endomorphism of $R$, which can be shown as in
\cite[Defition 4.4 (i)]{I}.  We also define $\eta^\opp(\si,\E_\si)$,
an extension of $\id\otimes\si^\opp$ to $R$ in a similar way.

\begin{theorem}
\label{qdim}
The category $\C'\cap\D$ is equivalent to
to the category of $R$-$R$ morphisms arising from decompositions
of $[\iota_{LR}][\lambda\otimes\id^\opp][\bar\iota_{LR}]$ where
$\la\in\Irr(\D)$.  We have $\dim(\C'\cap\D)=\dim \C\cdot\dim \D$.
\end{theorem}

\begin{proof}
If we have an irreducible half-braiding $\E_\si$
of an object $\si$ in $\D$ with
respect to $\C$, we have an extension $\eta(\si,\E_\si)$.  By definition,
we have $[\eta(\si,\E_\si)\cdot\iota_{LR}]=
[\iota_{LR}(\si\otimes\id^\opp)]$, so the extension
$\eta(\si,\E_\si)$ appears in the decomposition of
$[\iota_{LR}][\si\otimes\id^\opp][\bar\iota_{LR}]$.

Suppose an irreducible endomorphism $\rho$ of $R$ appears in
the decomposition of 
$[\iota_{LR}][\xi\otimes\id^\opp][\bar\iota_{LR}]$ for some
irreducible object $\xi\in\D$.  Then $\rho\cdot\iota_{LR}$
is contained in
\[ [\iota_{LR}][\xi\otimes\id^\opp][\bar\iota_{LR}][\iota_{LR}]
=\bigoplus_{\la\in\Irr(\C)}[\iota_{LR}][\xi\cdot\la\otimes\la^\opp],
\]
and Proposition \ref{pro} (2) shows that there exists an
object $\si\in\D$ satisfying
$[\rho\cdot\iota_{LR}]=
[\iota_{LR}(\si\otimes\id^\opp)]$, which means $\rho$
is an extension of $\si\otimes\id^\opp$ as an endomorphism.
We then have the following for some $U\in M\otimes M^\opp$.
\begin{align*}
\rho(a)&=(\si\otimes\id)(a),\quad a\in M\otimes M^\opp,\\
\rho(V)&=U^*V.
\end{align*}
We then have $U\in\Hom((\si\otimes\id^\opp)\cdot\Theta,
\Theta\cdot(\si\otimes\id^\opp))$ by a similar argument 
to the one in the middle of \cite[Page 141]{I}.  A further
argument similar to the one in \cite[Pages 141--143]{I}
shows that $U$ is of the form $U(\si,\E_\si)$ for some
half-braiding $\E_\si$ of $\si$ with respect to $\si$.

For the conjugate half-braiding, we have
$[\overline{\eta(\si,\E^\a_\si)}]=
[\eta(\bar\si,\E^{\bar\a}_{\bar\si})]$
as in \cite[Theorem 4.6 (iv)]{I}.

It is easy to see that the above correspondence indeed
gives equivalence of the two categories.

By what we have proved so far, $\dim(\C'\cap\D)$ is equal to
the square sum of the dimensions
of the irreducible $R$-$R$ sectors arising from decompositions
of $[\iota_{LR}][\lambda\otimes\id^\opp][\bar\iota_{LR}]$ where
$\la\in\Irr(\D)$.  The latter is then equal to
the square sum of the dimensions
of the irreducible $M\otimes M^\opp$-$M\otimes M^\opp$
sectors $\lambda\otimes\mu^\opp$ where
$\la\in\Irr(\D)$ and $\mu\in\Irr(\C)$, so we have the 
conclusion. 
\end{proof}

\section{The relative Drinfeld commutants arising from
$\alpha$-induction}

Now we change the notations and
let $\C$ be a modular tensor category realized as a full
subcategory of $\End(N)$ for a type III factor $N$.
Let $(\th,w,x)$ be a $Q$-system with $\th$ being an object in $\C$,
$M\supset N$ the corresponding subfactor and $\iota$ the inclusion
map $N\hookrightarrow M$.  We have $\alpha$-induction
$\a^\pm_\la$ for an object $\la$ in $\C$ as in
\cite{LR}, \cite{X}, \cite{BE1}, \cite{BE2}, \cite{BE3}, 
\cite{BEK1}, \cite{BEK2}, \cite{BEK3}.
Set $\D^\pm$ to be the fusion category generated by 
$\a_\la^\pm$ where $\la$ is an object of $\C$.
Set $\D$ to be the fusion category generated by
$\D^+$ and $\D^-$, and set $\D^0$ to be the fusion
category whose set of objects consists of those
which are objects of both $\D^+$ and $\D^-$.
Note that $\D$ is equal to the category generated by
$\iota\cdot\la\cdot\bar\iota$ for $\la\in\C$ by
\cite[Theorem 5.10]{BEK1}.
(The objects of $\D^0$ have been called ambichiral 
in \cite{BEK2}, and they correspond to
dyslectic/local modules in the terminology of
\cite{DMNO}, \cite{DNO}.)

The following result has been shown in \cite[Corollary 4.8]{BEK3}.
(Also see \cite[Corollary 3.30]{DMNO}, where unitarity is not
assumed.)
See \cite[Lemma 3.20, Remark 4.17]{BKLR} for an opposite braided
category.  Here $(\D^0)^\opp$ means a modular tensor category
with its braiding reversed.

\begin{theorem}
The Drinfeld center $(\D^+)'\cap \D^+$ of $\D^+$ is
equivalent to $\C\boxtimes (\D^0)^\opp$ as modular
tensor categories.
\end{theorem}

Now $\D^0$ is a full fusion subcategory of $\D^\pm$ and
$\D^\pm$ is a full fusion subcategory of $\D$.
We first compute $(\D^+)'\cap \D$ explicitly.
We use a relative braiding introduced in \cite[Proposition 3.12]{BE3}.
We remark that the arguments there use only the braided structure of
$\C$ and do not depend on a net structure or
locality (as noted in \cite[page 739]{BEK2}.)

For an object $\be$ in $\D$, we choose an isometry
$T\in\Hom(\be,\a^+_\nu \a^-_{\nu'})$ with some objects $\nu,\nu'$
in $\C$.  For any object $\la\in\C$, we set
\[
\E^+_{\a^+_\la}(\be)=T^*\e^+(\la,\nu\nu')\a^+_\la(T),
\]
as in \cite[(10)]{BEK3}.  (We have changed the notations slightly
here from those in \cite{BEK3}.)
By \cite[Lemma 3.1]{BEK3}, $\{\E^+_{\a^+_\la}(\be)\}$
gives a half-braiding of $\a^+_\la$ with respect to $\D$
and $\E^+_{\a^+_\la}(\be)$ does not depend on the choices of
$T$ and $\nu,\nu'$.  In particular, 
$\{\E^+_{\a^+_\la}(\be)\}_{\be\in \Irr(\D^+)}$ gives a half-braiding
of $\a^+_\la$ with respect to $\D^+$.
By \cite[Corollary 3.8]{BEK3}, we have the following.
(Note that \cite[Proposition 2.6]{BEK3} works here
since $\a^+_\la$ is an object of $\D^+$ rather than $\D$.)

\begin{proposition}
In the above setting,
we have $[\eta^\opp(\a_\la^+,\E_{\a^+_\la}^+)]=
[\eta(\a_{\bar \la}^+,\E_{\a^+_{\bar\la}}^+)]$.
\end{proposition}

Also, \cite[Lemma 3.7]{BEK3} gives the following about the
conjugate half-braiding $\bar\E^+_{\a^+_\la}(\be)$.

\begin{proposition}
We have $\bar\E^+_{\a^+_\la}(\be)=\E^+_{\a^+_\la}(\be)$.
\end{proposition}

We next follow the first paragraph of \cite[Section 4]{BEK3}.
Recall from \cite[Subsection 3.3]{BE3} that for
an object $\be_\pm$ in $\D^\pm$, the operators
\[
\E_r(\be_+,\be_-)=S^*\a_\mu^-(T)^*\e^+(\la,\mu)\a^+_\la(S)T
\] are unitaries in $\Hom(\be_+\be_-,\be_-\be_+)$
for objects $\la,\mu$ in $\C$ and isometries
$T\in \Hom(\be_+,\a_\la^+)$
and $S\in\Hom(\be_-,\a^-_\mu)$.  They do not depend on
the choices of $\la,\mu,S,T$.  They give a ``relative braiding''
between $\D^+$ and $\D^-$.
For an object $\tau$ in $\D^-$ and an object $\be$ in $\D^+$,
we put $\E^-_\tau(\be)=\E_r(\be,\tau)^*$, and this gives
a half-braiding $\{\E^-_\tau(\be)\}_{\be\in \Irr(\D^+)}$
of $\tau$ with respect to $\D^+$ by \cite[Lemma 4.1]{BEK3}.

The arguments similar to those
below \cite[(18)]{BEK3}  give the following.

\begin{proposition}
For $\tau,\tau'\in\Irr(\D^-)$, we have
$\Hom(\eta(\tau,\E^-_\tau),\eta(\tau',\E^-_{\tau'}))=
\de_{\tau,\tau'}\CC$.
\end{proposition}

Since we now assume $\C$ is a modular tensor category, 
\cite[Lemma 4.2]{BEK3} produces the following.

\begin{proposition}
For $\la,\mu\in\Irr(\C)$, we have
$\Hom(\eta(\a^+_\la,\E^+_{\a^+_\la}),\eta(\a^+_\mu,\E^+_{\a^+_\mu}))
=\de_{\la,\mu}\CC$.
\end{proposition}

In the same way as in \cite[Lemma 4.3]{BEK3}, we have the following.
(Note that we assumed $\tau\in {}_M\X_M^0$ in \cite[Lemma 4.3]{BEK3}
while we have $\tau\in\Irr(\D^-)$ here, but in the proof of
\cite[Lemma 4.3]{BEK3}, we used only the condition 
$\tau\in {}_M\X_M^-$ in the fourth line of the proof.)

\begin{proposition}
For $\la\in\Irr(\C)$ and $\tau\in\Irr(\D^-)$, we have
$\Hom(\eta(\a^+_\la,\E^+_{\a^+_\la}),\eta(\tau,\E^-_\tau))=
\de_{\la,\id}\de_{\tau,\id}\CC$.
\end{proposition}

Now \cite[Lemma 4.4]{BEK3} gives the following about the
conjugate half-braiding.

\begin{proposition}
We have $\bar\E^-_\tau(\be)=\E^-_\tau(\be)$ for objects $\be$ in
$\D^+$ and $\tau$ in $\D^-$.
\end{proposition}

We then have the following as in \cite[Theorem 4.6]{BEK3}.

\begin{proposition}
For $\la,\la'\in\Irr(\C)$ and $\tau,\tau'\in\Irr(\D^-)$,
we have
\[
\lan\eta(\a^+_\la,\E^+_{\a^+_\la})\eta(\tau,\E^-_{\tau}),
\eta(\a^+_{\la'},\E^+_{\a^+_{\la'}})\eta(\tau',\E^-_{\tau'})\ran
=\de_{\la,\la'}\de_{\tau,\tau'}\CC.
\]
\end{proposition}

Using Theorem \ref{qdim}, we know that the set 
$\{\eta^\opp(\a^+_\la,\E^+_{\a^+_\la})\eta(\tau,\E^-_\tau)\}$
with $\la\in\Irr(\C)$ and $\tau\in\Irr(\D^-)$ gives 
all representatives of the equivalence classes of
simple objects of $(\D^+)'\cap \D$.
We then have the following as a direct
analogue of \cite[Corollary 4.8]{BEK3} (along similar
arguments to those in \cite[page 18]{BEK3}, which
are based on \cite[Corollary 7.2]{I}).

\begin{theorem}
The fusion category $(\D^+)'\cap \D$ is equivalent to
$\C\boxtimes \D^-$.
\end{theorem}

\begin{remark}{\rm
Note that the above result has some formal
similarity to \cite[Theorem A]{H} in the appearance of $\D^-$.
}\end{remark}

We next consider $(\D^0)'\cap \D^+$.
In a way similar to the above, we have a half-braiding
$\{\E^+_\tau(\be)\}_{\be\in\Irr(\D^0)}$ with
respect to $\D^0$ for $\tau\in\Irr(\D^+)$.  We also have
a half-braiding $\{\E^-_\tau(\be)\}_{\be\in\Irr(\D^0)}$
with respect to $\D^0$ for $\tau\in\Irr(\D^0)$.
We similarly have the following.

\begin{proposition}The set 
$\{\eta(\tau,\E^+_\tau)\eta^\opp(\tau',\E^-_{\tau'})\}$
with $\tau\in\Irr(\D^+)$ and $\tau'\in\Irr(\D^0)$ gives 
all representatives of the equivalence classes of
simple objects of $(\D^0)'\cap \D^+$.
\end{proposition}

We then have the following again along similar
arguments to those in \cite[page 18]{BEK3}.

\begin{theorem}
The fusion category $(\D^0)'\cap \D^+$ is equivalent to
$\D^+\boxtimes \D^0$.
\end{theorem}

We next consider $(\D^0)'\cap \D$.
In a way similar to the above, we have a half-braiding
$\{\E^+_\tau(\be)\}_{\be\in\Irr(\D^0)}$ with
respect to $\D^0$ for $\tau\in\Irr(\D^+)$.  We also have
a half-braiding $\{\E^-_\tau(\be)\}_{\be\in\Irr(\D^0)}$
with respect to $\D^0$ for $\tau\in\Irr(\D^-)$.
Using Theorem \ref{qdim} and \cite[Theorem 4.2]{BEK2},
we have the following.

\begin{proposition}
The set 
$\{\eta(\tau,\E^+_{\tau})\eta(\tau',\E^-_{\tau'})\}$
with $\tau\in\Irr(\D^+)$ and $\tau'\in\Irr(\D^-)$ gives 
all representatives of the equivalence classes of
simple objects of $(\D^0)'\cap \D$.  
\end{proposition}

However, we do not
know whether $\eta(\tau',\E^-_{\tau'})$ is equivalent
to $\eta^\opp(\bar\tau',\E^-_{\bar\tau'})$ since
$\tau'$ is not in $\D^0$ in general.  So we now need an 
extra argument.

\begin{proposition}\label{inter}
For $\tau\in\Irr(\D^+)$ and $\tau'\in\Irr(\D^-)$,
the relative braiding $\E_r(\tau,\tau')$ gives an element
in $\Hom(\eta(\tau,\E^+_{\tau})\eta(\tau',\E^-_{\tau'}),
\eta(\tau',\E^-_{\tau'})\eta(\tau,\E^+_{\tau})).$
\end{proposition}

\begin{proof}
We show that $\E_r(\tau,\tau')$ gives an intertwiner
on the level of half-braiding.
Then the arguments in the proof of \cite[Proposition 3.12]{BE3}
based on \cite[Lemma 3.25]{BE1} give the desired conclusion.
\end{proof}

We also have the following.

\begin{proposition}
For (possibly reducible) objects $\tau,\tau'$ of $\D^+$, we have
$\Hom(\eta(\tau,\E^+_{\tau}),\eta(\tau',\E^+_{\tau'}))=
\Hom(\tau,\tau')$.
\end{proposition}

\begin{proof}
It is easy to see that the right hand side is contained in the
left hand side.  The dimensions of the both hand sides are equal,
so we have the equality.
\end{proof}

We also have a similar equality for $\eta(\tau,\E^-_{\tau})$
for an object of $\D^-$.

Now consider the fusion category whose irreducible objects are given by
$\{\eta(\tau,\E^+_{\tau})\eta(\tau',\E^-_{\tau'})\}$ with
$\tau\in\Irr(\D^+)$ and $\tau'\in\Irr(\D^-)$.  We 
consider the $6j$-symbols of this fusion category.  Then it splits
as a product of two $6j$-symbols as in Fig. \ref{q6j}.
In this diagram, we follow the graphical convention of
\cite[Section 3]{BEK1}.  In particular, we
compose morphisms from the top to the bottom.
The thick wire represents an irreducible object of the
fusion category $\Irr(\D^+)$ and the thin wires represents
an irreducible object of the fusion category $\Irr(\D^-)$.
The inner products in Fig. \ref{q6j} represents those between
two intertwiners and the compositions give complex numbers.
When we compose two irreducible morphisms, we switch two
components using Proposition \ref{inter} and this give a
relative braiding $\E_r$ at the upper left corner of 
Fig. \ref{q6j}.  All the crossings in  Fig. \ref{q6j}
represent the relative braiding $\E_r$ or its conjugate.

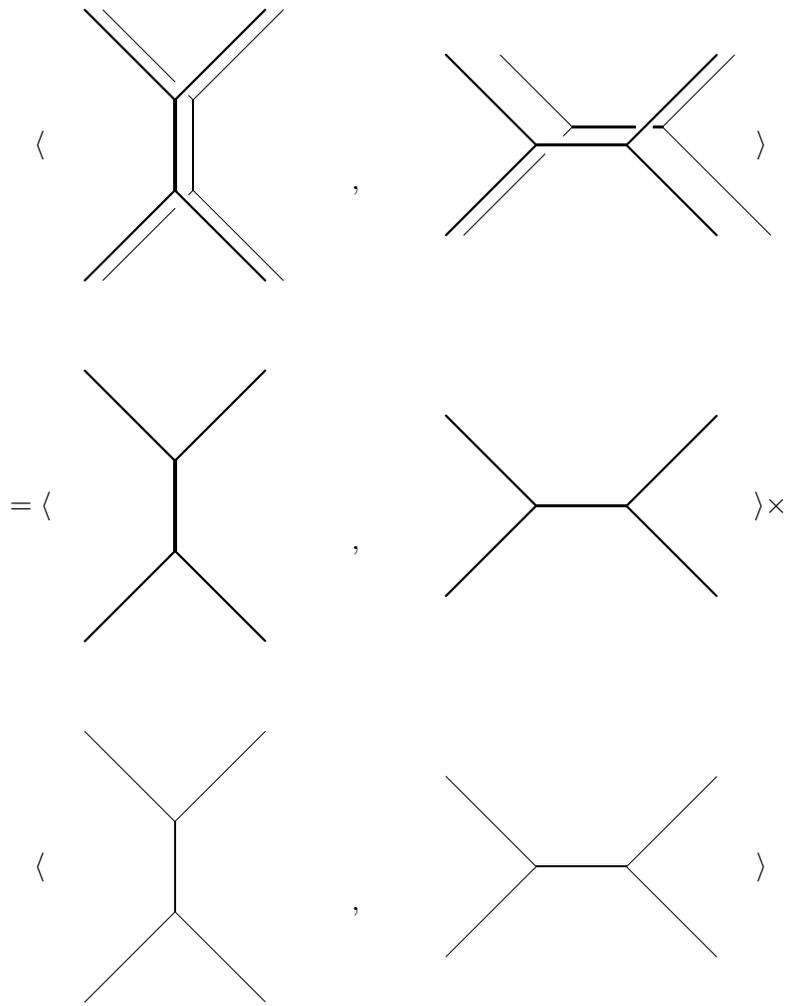
\begin{figure}[htb]
\begin{center}
\unitlength 1.2mm
\begin{picture}(90,130)
\thicklines
\put(10,90){\line(1,1){10}}
\put(20,100){\line(1,-1){10}}
\put(20,100){\line(0,1){10}}
\put(20,110){\line(-1,1){10}}
\put(20,110){\line(1,1){10}}
\put(50,95){\line(1,1){10}}
\put(50,115){\line(1,-1){10}}
\put(60,105){\line(1,0){10}}
\put(70,105){\line(1,1){10}}
\put(70,105){\line(1,-1){10}}
\put(10,50){\line(1,1){10}}
\put(20,60){\line(1,-1){10}}
\put(20,60){\line(0,1){10}}
\put(20,70){\line(-1,1){10}}
\put(20,70){\line(1,1){10}}
\put(50,55){\line(1,1){10}}
\put(50,75){\line(1,-1){10}}
\put(60,65){\line(1,0){10}}
\put(70,65){\line(1,1){10}}
\put(70,65){\line(1,-1){10}}
\thinlines
\put(10,10){\line(1,1){10}}
\put(20,20){\line(1,-1){10}}
\put(20,20){\line(0,1){10}}
\put(20,30){\line(-1,1){10}}
\put(20,30){\line(1,1){10}}
\put(50,15){\line(1,1){10}}
\put(50,35){\line(1,-1){10}}
\put(60,25){\line(1,0){10}}
\put(70,25){\line(1,1){10}}
\put(70,25){\line(1,-1){10}}
\put(12,90){\line(1,1){8}}
\put(22,100){\line(-1,-1){0.5}}
\put(22,100){\line(1,-1){10}}
\put(22,100){\line(0,1){10}}
\put(22,110){\line(-1,1){0.5}}
\put(12,120){\line(1,-1){8}}
\put(22,110){\line(1,1){10}}
\put(52,95){\line(1,1){9}}
\put(64,107){\line(-1,-1){1}}
\put(64,107){\line(-1,1){8}}
\put(64,107){\line(1,0){7}}
\put(74,107){\line(-1,0){1}}
\put(74,107){\line(1,1){8}}
\put(74,107){\line(1,-1){12}}
\put(5,25){\makebox(0,0){$\langle$}}
\put(4,65){\makebox(0,0){$=\langle$}}
\put(5,105){\makebox(0,0){$\langle$}}
\put(85,25){\makebox(0,0){$\rangle$}}
\put(86,65){\makebox(0,0){$\rangle\times$}}
\put(85,105){\makebox(0,0){$\rangle$}}
\put(40,100){\makebox(0,0){$,$}}
\put(40,60){\makebox(0,0){$,$}}
\put(40,20){\makebox(0,0){$,$}}
\end{picture}
\end{center}
\caption{$6$j-symbols}
\label{q6j}
\end{figure}

We thus obtain the following theorem.

\begin{theorem}
The fusion category $(\D^0)'\cap \D$ is equivalent to
$\D^+\boxtimes \D^-$.
\end{theorem}

We show some example now.  A typical appearance of
$\a$-induction is an extension of a completely rational
local conformal net in the sense of
\cite[page 498]{KL}, \cite[Definition 8]{KLM},
\cite[Definition 3.1]{LR}.  Note that strong
additivity and split property in the definition
of complete rationality \cite[Definition 8]{KLM}
are unnecessary due to \cite{LX} and \cite{MTW},
respectively.  Let $\{\A(I)\subset\B(I)\}$ be
such an extension, where $I$ is an interval
contained in $S^1$.  Let $\C$ be the representation
category of the local conformal net $\{\A(I)\}$
and consider the $\a$-induction for a subfactor
$\A(I)\subset\B(I)$ for some interval $I$ as in
\cite[Definition 3.3]{BE1}.  Then we have
$\D^0,\D^\pm, \D$ from this $\a$-induction as in
\cite{BEK1}, so the above results apply.  Note 
that $\D^0$ is the representation category of 
$\{\B(I)\}$ and $\D^\pm$ are the categories of
soliton sectors.

Consider an extension of a completely rational
local conformal net arising from a conformal
embedding $SU(2)_{10}\subset SO(5)_1$ as in
\cite[Example 2.2]{BE2}.  In this case,
the category $\C$ has 11 simple objects,
and $\D^0,\D^+,\D^-,\D$ have 3, 6, 6, and 12
simple objects, respectively, as in \cite[Fig. 2]{BE3}.
This setting gives concrete examples to which the
above results apply.

\end{document}